\documentclass[11pt,amsfonts, epsfig]{amsart}
\usepackage[left=3cm,right=3cm,top=2cm,bottom=2.5cm,a4paper]{geometry}
\usepackage{amsmath, amscd, amssymb}
\usepackage{tikz-cd}
\usepackage[all,cmtip]{xy}
\usepackage{graphpap, color}
\usepackage[mathscr]{eucal}
\usepackage{mathrsfs}
\usepackage{pstricks}
\usepackage{color}
\usepackage{cancel}
\usepackage{stmaryrd}
\usepackage{verbatim}

\usepackage{hyperref}

\def\defeq{:=}

\def\ff{{\mathfrak f}}

\numberwithin{equation}{section}
\def\sP{{\mathscr P}}

\def\sO{{\mathscr O}}
\def\sM{{\mathscr M}}

\def\sL{{\mathscr L}}

\def\sO{\mathscr{O}}

\newcommand{\CC}{\mathbb{C}}
\newcommand{\EE}{\mathbb{E}}
\newcommand{\LL}{\mathbb{L}}
\newcommand{\PP}{\mathbb{P}}
\newcommand{\QQ}{\mathbb{Q}}

\newcommand{\VV}{\mathbb{V}}

\newcommand{\TT}{\mathbb{T}}

\def\fR{{\mathfrak R}}

\def\sP{{\mathscr P}}

\def\Tot{\mathrm{Tot}}

\def\gst{\text{gst}}

\def\cA{{\mathcal A}}

\newcommand{\cal}{\mathcal}

\def\cA{{\cal A}}
\def\cB{{\cal B}}
\def\cC{{\cal C}}
\def\cD{{\cal D}}
\def\cE{{\cal E}}

\def\cH{{\cal H}}

\def\cL{{\cal L}}

\def\cO{{\cal O}}
\def\cP{{\cal P}}

\def\cU{{\cal U}}
\def\cV{{\cal V}}

\def\cX{{\cal X}}
\def\cY{{\cal Y}}
\def\cZ{{\cal Z}}

\def\fC{\mathfrak{C}}

\def\ff{\mathfrak{f}}

\def\mapright#1{\,\smash{\mathop{\lra}\limits^{#1}}\,}

\def\dual{^{\vee}}

\def\sta{^\ast}

\def\virt{^{\mathrm{vir}}}

\def\sta{^{\ast}}

\def\sta{^*}

\def\lra{\longrightarrow}
\def\lpri{_{\mathrm{pri}}}
\def\lsta{_{\ast}}

\newcommand{\si}{\sigma}

\def\begeq{\begin{equation}}
\def\endeq{\end{equation}}
\def\and{\quad{\rm and}\quad}
\def\bl{\bigl(}
\def\br{\bigr)}
\def\defeq{:=}

\def\sub{\subset}

\def\and{\quad\text{and}\quad}

\DeclareMathOperator{\image}{Im}

 \DeclareMathOperator{\rank}{rank}
 
 \DeclareMathOperator{\rede}{{red}}

\newtheorem{prop}{Proposition}[section]
\newtheorem{theo}[prop]{Theorem}
\newtheorem{lemm}[prop]{Lemma}
\newtheorem{coro}[prop]{Corollary}
\newtheorem{rema}[prop]{Remark}

\newtheorem{defi}[prop]{Definition}

\newtheorem{defi-prop}[prop]{Definition-Proposition}

\def\Ob{\cO b}

\def\lloc{_{\mathrm{loc}}}

\def\loc{{\mathrm{loc}}}

\def\bul{^\bullet}

\def\sta{^\ast}
\def\image{\text{Im}\,}

\def\sO{{\mathscr O}}

\def\beq{\begin{equation}}
\def\eeq{\end{equation}}

\def\vsp{\vskip5pt}

\def\fM{{\mathfrak M}}

\def\bee{\begin{equation}}
\def\eeq{\end{equation}}

\def\ti{\tilde}

\def\LL{{\mathbb L}}

\def\lgst{_{\rm gst}}

\def\gst{{\rm{gst}}}

\newcommand{\wtil}{\widetilde}
\newcommand{\Gysin}[1]{0_{#1}^!}
\newcommand{\vphi}{\varphi}

\title[]{Splitting of the virtual class for genus one stable quasimaps
}

\date{}

\author{Sanghyeon Lee}
\address{Korea Institute for Advanced Study(KIAS), 85 Hoegiro, Dongdaemun-gu, Seoul 02455, Republic of Korea}
\email{sanghyeon@kias.re.kr}

\author[Mu-Lin Li]{Mu-Lin Li}
\address{School of Mathematics, Hunan University, China} \email{mulin@hnu.edu.cn}

\begin{document}

\maketitle
\setcounter{tocdepth}{2}

\begin{abstract}
We analyse the local structure of moduli space of genus one stable quasimaps. Combining it with the p-fields theory developed in \cite{L}, we prove the splitting formula for the virtual cycle of stable quasimaps to complete intersections in $\PP^n$.
\end{abstract}
\tableofcontents

\section{Introduction}

The moduli space of stable quasimaps to arbitrary GIT quotient is a generalization of the moduli space of stable quotient defined by Marian, Oprea and Pandharipande \cite{MOP}, which was constructed and studied by Ciocan-Fontanine, Kim  and Maulik \cite{FKM}.
When the target is a projective complete intersection, Ciocan-Fontanine and Kim \cite{FK2} proved that the invariants of stable quasimaps can be related to the Gromov-Witten invariants by mirror map for all genus (see also \cite{FK}, \cite{FK0}, \cite{Zh} for different cases, and \cite{CJR1}, \cite{CJR2} for different proofs). The genus zero stable quasimap (stable quotient) invariants of complete intersections are computed by Cooper and Zinger \cite{CZ}, and Ciocan-Fontanine and Kim \cite{FK}. Kim and Lho \cite{Kiml} calculate the genus one invariants of complete intersection without markings by using infinitesimal marked points.

Let $X=(q_1(x)=\cdots=q_m(x)=0)\subset \PP^n$ be a smooth complete intersection. Let $Q_{g,k}(X,d)$ be the moduli stack of genus $g$ stable quasimaps to $X$ with degree $d$  and $k$ markings. It is a
proper Deligne Mumford (DM for short)-stack, and carries a canonical virtual cycle $[Q_{g,k}(X,d))]\virt$.

Especially, for $k=1$,
$X=\PP^n$ case, $\cX:=Q_{1,1}(\PP^n,d)$ has two smooth components by Theorem \ref{Main3}. One is the main component $\cX_{\rede}$,  and the other component is  the ghost component $\cX_{\gst}$. Let $\pi_{\cX}: \cC_{\cX}\to \cX$ be the universal family, and $\sL_{\cX}$ be the universal line bundle over $\cC_{\cX}$. Then the restriction $\pi_{\cX*}\sL^{\otimes r}_{\cX}|_{\cX_{\rede}}$ is locally free for all all positive integers $r$.
In \eqref{red-def}, we define the reduced virtual cycle $A^{\rede}_{1,d}$ by the refined euler class of the bundle $\pi_{\cX*}\sL^{\otimes r}_{\cX}|_{  \cX_{\rede} }$.
Then we have the following splitting formula for virtual cycle.
\begin{theo}\label{main-0} Let $X=(q_1(x)=\cdots=q_m(x)=0)\subset \PP^n$ be a smooth complete intersection, then
\begin{eqnarray*}
[Q_{1,1}(X,d))]^{\virt}&=& A^{\rede }_{1,d}+ \\
&& (-1)^{(\sum_i\deg q_i)d} \, \iota_* \left( \frac{c(\cH\dual \boxtimes ev_1^* T_X)  }{c(\cH\dual \boxtimes L_2) } \right)_{n - m-1} \cap \left( [\overline{M}_{1,1}] \times [Q_{0,2}(X,d)]\virt \right)
\end{eqnarray*}
where $\iota : \overline{M}_{1,1} \times Q_{0,2}(X,d) \to Q_{1,1}(X,d)$ is the node-identifying morphism, $\cH$ is the Hodge bundle over $\overline{M}_{1,1}$,  $L_2$ is the universal tangent bundle over $Q_{0,2}(\PP^n,d)$ at the second marked point, which comes from splitting of the node and $A^{\rede }_{1,d}$ is the reduced virtual cycle defined by (\ref{red-def}).
\end{theo}

Let $\psi$ be the psi-class of $Q_{1,1}(X,d)$ at the marked point. For $\gamma \in H^{2k}(X,\QQ)$, $k \le 1$, we can define the following stable quasimap invariants
$$
\langle\psi^{a}ev^*\gamma\rangle_{1,1,d}:=\int_{[Q_{1,1}(X,d)]^{\text{vir}}}\psi^{a}ev^*\gamma,
$$
when $a+k=\text{vdim}\, Q_{1,1}(X,d)$.

The reduced genus one invariants of stable quasimaps to smooth complete intersection $X\subset\PP^n$ is defined as follows
\begin{defi}
\beq\label{def-red}
\langle\psi^{a}ev^*\gamma\rangle^{\rede}_{1,1,d}:=\int_{A_{1,d}^{\rede}} \psi^{a}ev^*\gamma.
\eeq
\end{defi}

Then we prove the following equality as formula (\ref{form-2}) in the paper,
\beq\label{def-red-2}
\langle\psi^{a}ev^*\gamma\rangle^{\rede}_{1,1,d}=\int_{\cX_{\rede}}\psi^{a}ev^*\gamma\cup e^{\mathrm{ref}}\left( \oplus_{i=1}^m\pi_{\cX*}\sL^{\otimes \deg q_i}_{\cX}|_{\cX_{\rede}} \right).
\eeq

This reduced invariants can be calculated by using the localization formula similarly as Zinger \cite{Zin1} did in genus one Gromov-Witten invariants, and as the second author \cite{MLL} did in genus one stable quasimap invariants without marking. We have the following formula which connect the reduced and standard stable quasimap invariants for complete intersections
\begin{coro}\label{coro-1}
Let $X=(q_1(x)=\cdots=q_m(x)=0)\subset \PP^n$ be a smooth complete intersection. For $\gamma \in H^{2k}(X,\QQ)$ where $k \le 1$, we have
\begin{eqnarray*}
\langle\psi^{a}ev^*\gamma \rangle_{1,1,d}&=&\langle\psi^{a}ev^*\gamma \rangle^{\rede}_{1,1,d}-\frac{1}{24}\bigg(\int_{Q_{0,2}(X,d)}\psi^{a}ev_1^*\gamma \cup c_{n-m-2}(ev_2^* T_X)\\
&&-(n-m-1)\int_{Q_{0,2}(X,d)}\psi^{a}ev^*\gamma \, \psi_2^{n-m-2}\bigg),
\end{eqnarray*}
where $a+k=\text{vdim}\, Q_{1,1}(X,d)$. Furthermore, if $X$ is a Calabi-Yau threefold, then $c_1(T_X)=0$, and
\begin{eqnarray}
\langle\psi^{a}ev^*\gamma\rangle_{1,1,d}&=&\langle\psi^{a}ev^*\gamma\rangle^{\rede}_{1,1,d}+\frac{1}{12}\int_{Q_{0,2}(X,d)}\psi^{a}ev^*\gamma \, \psi_2. \nonumber
\end{eqnarray}
\end{coro}
The term $\langle\psi_1\rangle^{\rede}_{1,1,d}$ plays an important role in Oh and the authors' splitting formula \cite{LLO} for genus two stable quasimap invariant of complete intersection Calabi-Yau threefolds in $\PP^n$. Thus this paper can be seen as the first step in our approach to the calculation of genus two stable quasimap invariants.
%
%

\medskip

\noindent{\bf Acknowledgment}:   The second author thanks Huai-Liang Chang, Bumsig Kim,  Jun Li, and A. Zinger  for helpful discussions. This work was supported by the Start-up Fund of Hunan University. The first author thanks Jeongseok Oh for helpful discussions. This work was supported by a KIAS Individual Grant MG070902 at Korea Institute for Advanced Study.

\section{Local charts and local equations}\label{sect:local}

\subsection{Relative obstruction theories of quasi-map spaces}\label{Sect:relPOTs}

Here we introduce relative perfect obstruction theories of the quasi-map space $Q_{1,k}(\PP^n,d)$ and the quasi-map space with fields $Q_{1,k}(\PP^n,d)^p$.
We introduce some Artin stacks, which will be used as bases of the relative perfect obstruction theories. Let $\fM_{1,k}$ be the Artin stack of nodal curves of
genus one with $k$-markings.




\begin{defi}\label{defstack1}
Let $\, \fM^{wt}_{1,k,d}$ be the groupoid associating each scheme $S$ to the set $\, \fM^{wt}_{1,k,d}(S) = (\mathcal{C}_S, \{p_j:S\to\cC_S\}_{j=1}^k)$ where $(\pi:\mathcal{C}_S \rightarrow S,\chi)$ is a flat family of prestable genus one weighted nodal curves with $k$ marked points. We will usually abbreviate it by $\fM^{wt}_{1,k}$.
\end{defi}

\begin{defi}\label{defstack}
Let $\, \fM^{line}_{1,k}$ be the groupoid associating each scheme $S$ to the set $\, \fM^{line}_{1,k}(S)=(\mathcal{C}_S,\{p_j:S\to\cC\}_{j=1}^k,\mathscr{L})$, where $\pi:\mathcal{C}_S \rightarrow S$ is a flat family of connected genus one nodal curves
and $\{\mathscr{L}\}$ is a line bundle on $\mathcal{C}_S$ of degree $d$ along fibers of $\mathcal{C}_S/S$. An arrow from
$(\mathcal{C}_S,\{p_j:S\to\cC_S\}_{j=1}^k,\sL)$ to $ (\mathcal{C}^{\prime}_S,\{p^{\prime}_j:S\to\cC^{\prime}\}_{i=1}^k,\mathscr{L}')$ consists of $f:\mathcal{C}_S\rightarrow\mathcal{C}^{\prime}_S$ and an isomorphism $\theta_f:f^{*}\mathscr{L}^{\prime}\rightarrow \mathscr{L}$, which preserve the markings and the sections.
\end{defi}

Let $(C,\{p_j\}_{j=1}^k, D)$ be the  $k$-pointed (connected) nodal
elliptic curves $C$ with effective divisors $D\subset C$ supported on the smooth loci of $C$.
Then $(C,\{p_j\}_{j=1}^k, D)$ is stable if the induced weighted nodal curve $(C,\{p_j\}_{j=1}^k, \deg D)$ is stable.

\begin{defi}\label{defstack2}
Let $\, \fM_{1,k,d}^{div}$ be the groupoid associating each scheme $S$ to the set $\, \fM_{1,k,d}^{div}(S) = (\mathcal{C}_S,\{p_j:S\to\cC\}_{j=1}^k,\mathscr{D})$, where $\pi:\mathcal{C}_S \rightarrow S$ is a flat family of connected stable genus one nodal curves and $\mathscr{D}$ is an effective divisor on $\mathcal{C}_S$ whose degree is $d$ on each fiber.
\end{defi}
Note that $\fM_{1,k}$, $\fM^{wt}_{1,k,d}$, $\fM^{line}_{1,k,d}$ and $\fM_{1,k,d}^{div}$ are smooth Artin stacks. The morphism $\fM_{1,k,d}^{div} \to \fM^{wt}_{1,k,d}$ is smooth and proper with connected fibers, and the morphism $\fM^{wt}_{1,k}\to \fM_{1,k}$ is \'etale.
The natural (dual) relative obstruction theory of $Q_{1,k}(\PP^n,d)$ over $\fM^{line}_{1,k}$ is defined by
\begin{align}\label{eq:relPOT1}
\mathbb{E}\dual_{Q_{1,k}(\PP^n,d) / \fM^{line}_{1,k}} := R\pi_* \sL_{\cC}^{\oplus n},
\end{align}
where $\pi : \cC \to Q_{1,k}(\PP^n,d)$ is the universal curve and $\sL_{\cC}$ is the universal bundle over $\cC$, which coincides with the pull-back of the universal bundle $\sL$ over $\fM^{line}_{1,k}$ via the forgetful morphism $\mathfrak{f}:Q_{1,k}(\PP^n,d) \to \fM^{line}_{1,k}$.

Next we consider a relative obstruction theory of $Q_{1,k}(\PP^n,d)$ over $\fM_{1,k}$.
The morphism $\fM^{line}_{1,k} \to \fM^{wt}_{1,k}$ is given by associating $(C,L)$ to the weight on $C$, given by the degree of the line bundle $L$ restricted on each irreducible component of $C$. Note that this morphism is smooth.
Hence the morphism $\fM^{line}_{1,k} \to \fM_{1,k}$ is smooth. Hence there is a natural relative obstruction theory $\EE_{ Q_{1,k}(\PP^n,d) / \fM_{1,k} }$ to $  \LL_{ Q_{1,k}(\PP^n,d) / \fM_{1,k} }$, which is induced from the relative obstruction theory $\EE_{ Q_{1,k}(\PP^n,d) / \fM^{line}_{1,k} } \to \LL_{ Q_{1,k}(\PP^n,d) / \fM^{line}_{1,k} }$ \cite[Proposition 7.2]{BF}.

From the definition of relative obstruction theories and octahedral axiom of derived categories, there is a natural distinguished triangle:
\begin{align*}
& \mathfrak{f}^*\TT_{\fM^{line}_{1,k} / \fM_{1,k} }[-1] \to \EE\dual_{Q_{1,k}(\PP^n,d) / \fM^{line}_{1,k} } \to \EE\dual_{Q_{1,k}(\PP^n,d) / \fM_{1,k} }\stackrel{+1}{\lra}
\end{align*}

which fits in to the commutative diagram of distinguished triangles:
\begin{align}\label{eq:POTdiag1}
\xymatrix{ \mathfrak{f}^*\TT_{\fM^{line}_{1,k} / \fM_{1,k} }[-1] \ar[r]^-{\varphi} &\EE\dual_{Q_{1,k}(\PP^n,d) / \fM^{line}_{1,k} } \ar[r] &\EE\dual_{Q_{1,k}(\PP^n,d) / \fM_{1,k} } \ar[r]^-{+1} &  \\
\mathfrak{f}^*\TT_{\fM^{line}_{1,k} / \fM_{1,k} }[-1] \ar[r]^-{a} \ar@{=}[u]& \TT_{Q_{1,k}(\PP^n,d) / \fM^{line}_{1,k} } \ar[r] \ar[u]^-{b} & \TT_{Q_{1,k}(\PP^n,d) / \fM_{1,k} }\ar[r]^-{+1} \ar[u] &
}
\end{align}

On the other hand, in a similar manner as in  \cite[Lemma 2.8]{CL} we have the following commutative diagram of distinguished triangles:

\begin{align}\label{eq:POTdiag2}
\xymatrix{ R^\bullet\pi_* \cO_{Q_{1,k}(\PP^n,d)} \ar[r]^-{\varphi'} & R^\bullet\pi_* \sL_{\cC}^{\oplus n+1} = \EE\dual_{Q_{1,k}(\PP^n,d) / \fM^{line}_{1,k} } \ar[r] & R^\bullet\pi_* f^* T_{ [\CC^{n+1} / \CC^*] } \ar[r]^-{+1} &  \\
\mathfrak{f}^*\TT_{\fM^{line}_{1,k} / \fM_{1,k} }[-1] \ar[r]^-{a} \ar[u]^-{\cong} & \TT_{Q_{1,k}(\PP^n,d) / \fM^{line}_{1,k} } \ar[r] \ar[u]^-{b} & \TT_{Q_{1,k}(\PP^n,d) / \fM_{1,k} }\ar[r]^-{+1} \ar[u] &
}
\end{align}
where $\pi : \cC_{Q_{1,k}(\PP^n,d)} \to Q_{1,k}(\PP^n,d)$ is the universal curve and $\sL_{\cC}$ is the universal bundle over $\cC_{Q_{1,k}(\PP^n,d)}$, $f : \cC_{Q_{1,k}(\PP^n,d)} \to [\CC^{n+1}/\CC^*]$ is a universal morphism induced from the universal section $(u_0,\dots,u_n)$ of $\sL_{\cC}^{n+1}$. Note that the (pull-back of) the tangent complex $T_{\CC^{n+1} / \CC^*}$ of the quotient stack is the complex
\[
\xymatrix@C=45pt{
\cO_{\CC^{n+1}} \ar[r]^-{  (x_0,\dots,x_{n} ) } & \cO_{\CC^{n+1}}^{n+1}
}
\]
where $x_0,\dots x_{n}$ is the coordinate functions of $\CC^{n+1}$. Note that the distinguished triangle on the first horizontal arrow is obtained from the exact sequence
\[
\xymatrix{
0 \ar[r] & \cO_{Q_{1,k}(\PP^n,d) } \ar[rr]^-{ (u_0,\dots,u_{n})  } & & \sL_{\cC}^{n+1} \cong \sL_{\cC} \times \CC^{n+1} \ar[r] & f^*T_{[\CC^{n+1} / \CC^*] } \ar[r] & 0
}
\]
by taking the pull-back and the pushforward.
Then we have
\begin{align*}
\EE\dual_{Q_{1,k}(\PP^n,d) / \fM_{1,k}} & \cong \mathrm{cone} \left( u^*\TT_{\fM^{line}_{1,k} / \fM_{1,k} } \stackrel{b\circ a}{\lra} \EE\dual_{Q_{1,k}(\PP^n,d) / \fM^{line}_{1,k} } \right) \\
& \cong \mathrm{cone} \left( R^\bullet\pi_* \cO_{Q_{1,k}(\PP^n,d)} \stackrel{\varphi'}{\lra} R^\bullet\pi_*f^*T_{[\CC^{n+1} / \CC^*]} \right) \\
& \cong R^\bullet\pi_* f^* T_{[\CC^{n+1} / \CC^*]}.
\end{align*}

\begin{rema}\label{distingtriRemark}
By the above argument, we can replace $\varphi$ by $\varphi'$, which is the morphism induced from the section $  (u_0,\dots, u_{n})  : \cO_{Q_{1,k}(\PP^n,d)} \to \sL_{\cC}^{n+1}$ by taking the derived pushforward.
\end{rema}

Next we define a local relative obstruction theory of $Q_{1,k}(\PP^n,d)$ over $\fM_{1,k,d}^{div}$. Although there is no natural morphism from $Q_{1,k}(\PP^n,d)$ to $\fM_{1,k,d}^{div}$, we can consider the morphism locally as follows.
Consider a point $x=[(C,p_1,\dots,p_k,L,\{u_i\}^{n}_{i=0})] \in Q_{1,k}(\PP^n,d)$ and an open subset $\cU_0 \subset Q_{1,k}(\PP^n,d)$ defined by the condition $u_0 \neq 0$ containing $x$. Then there is a morphism $p: \cU_0 \to \fM_{1,k,d}^{div}$ defined by
\begin{align*}
p : \cU_0 & \to \fM_{1,k,d}^{div}, \\
[(C,p_1,\dots,p_k,L,\{u_i\}^{n}_{i=0})] & \mapsto [(C, p_1,\dots,p_k, u_0^{-1}(0) )].
\end{align*}

Over this local chart $\cU_0$ of $Q_{1,k}(\PP^n,d)$, a (dual) relative obstruction theory $\EE\dual_{\cU_0 / \fM_{1,k,d}^{div}}$ is defined by the following in \cite{CL}:

\begin{align}\label{eq:relPOT3}
\EE\dual_{\cU_0 / \fM_{1,k,d}^{div}} : = R\pi_* \cO_{\cC}(\cD)^{\oplus n}
\end{align}
where $\pi : \cC \to \cU_0$ is the universal curve and $\cD \subset \cC$ is the universal divisor defined by the universal section $s_0$ of the universal bundle $\sL_{\cC}$ on $\cC$.

\subsection{Local charts and local equations}
In this section, we will study the local structure of $Q_{1,k}(\PP^{n},d)$, parallel to \cite{HL10} which studied local structure of the stable map space $M_{1,k}(\PP^n,d)$.



Recall the the morphism from the open neighbourhood $\cU_0 \subset Q_{1,k}(\PP^n,d)$ to the Artin stack $\fM_{1,k,d}^{div}$ defined in Section \ref{Sect:relPOTs}.
We also consider a closed point
$$x=[(C,p_1,\dots,p_k,L,\{u_i\}^{n}_{i=0})] \in \cU_0.$$
Let us denote the divisor $u_0^{-1}(0)$ by $D$ and let $\cV \to \fM_{1,k,d}^{div}$ be a smooth affine chart with
$$[( (\cC_{\cV})_0, p_1(0),\dots, p_j(0), \cD_0)] = [ (C, p_1,\dots, p_k, D ) ] = q(x).$$
Here, $\cC_{\cV}$ is a canonical curve over $\cV$,  $p_i : \cV \to \cC_{\cV}$ are universal sections and $\cD$ is a universal divisor on $\cC_{\cV}$.
In fact, $\cU_0$ will be turned out as an open set of a total space of  $\rho_* \cO_{\cC_{\cV}}(\cD)$ where $\cD$ is a universal divisor on the universal curve $\rho : \cC_{\cV} \to \cV$. So we need to find a resolution of $\rho_* \cO_{\cC_{\cV}}(\cD)$.
For this, we first show the following lemma.

\begin{lemm}\label{divisor splitting}
By taking $\cV$ small enough, there is an equivalence of line bundles:
\[
\cO_{\cC_{\cV}}(r \cD) \cong \cO_{\cC_{\cV}}(\cD_1 + \dots + \cD_{rd})
\]
where $r\geq 1$ is an integer, $\cD_1 \dots \cD_{rd}$ are sections $\cV \to \cD$ disjoint to each others.
\end{lemm}
\begin{proof}[(Sketch of the proof)] Basically the proof can be obtained similarly as \cite[Lemma 2.1]{LLO}.
\\
\medskip
Case 1) $d=1$. It is clear that there is nothing to proof. So we will just sketch the proof.
\\
\medskip
Case 2) $d \geq 2$.
Take the neighbourhood $\cV$ small enough. Then, from the degree condition $d\geq 2$, we can find two sections $s_1, s_2$ of $\cO_{\cC_{\cV}}(r \cD)$ which gives a family of degree $r \cdot d$ morphisms to $\PP^1$.  Since $\cV$ is small enough, we can find a linear combination $a s_1 + b s_2$ whose zero is $\cD_1 + \dots \cD_{rd}$ where $D_i$ are family of degree 1 effective divisors disjoint to each others.
\end{proof}





Same as the stable map spaces case \cite{HL10}, We can choose sections $\cA, \cB : \cV \to \cC_{\cV}$ lies in core subcurves for each fiber, and disjoint with each others. Moreover we may assume that $\cA,\cB$ are disjoint to the divisors $\cD_1, \dots, \cD_{rd}$. Here, we define core subcurve of a genus $g$ curve $X$ by a minimal genus $g$ subcurve of $X$.

Let $\sL := \cO_{\cC_{\cV}}(\cD)$. By the above lemma, we have $\sL^{\otimes r} \cong \cO_{\cC_{\cV}}(\cD_1 + \dots \cD_{rd})$.
We consider the inclusion of
sheaves
$$ \sM_i:=\cO_{\cC_{\cV}}(\cD_i+\cA-\cB)\mapright{\sub}
\sM=:\cO_{\cC_{\cV}}\left( \sum_{i=1}^{rd}\cD_i +\cA-\cB \right)
$$
and the induced inclusions
$$\eta_i: \rho_*\sM_i \hookrightarrow
\rho_*\sM.
$$
Both are locally free since $R^1\rho\lsta\sM_i$ and
$R^1\rho\lsta\sM=0$. By Riemann-Roch,
$\rho\lsta\sM_i$ is invertible and the rank of $\rho\lsta\sM$ is
$d$. We then let
$$
\varphi: \rho\lsta\sM\lra
\rho\lsta\bl\cO_{\cC_{\cV}}\left( \sum_{i=1}^{rd}\cD_i +\cA-\cB \right)|_{\cA}\br=\rho\lsta\bl\sO_{\cA}(\cA))
$$
and
$$\varphi_i: \rho\lsta\sM_i\lra
\rho\lsta\bl\cO_{\cC_{\cV}}\left( \sum_{i=1}^{rd}\cD_i +\cA-\cB \right)|_{\cA}\br=\rho\lsta\bl\sO_{\cA}(\cA))
$$
be the evaluation homomorphisms. Obviously, $\varphi_i=\varphi
\circ \eta_i$. Since we assumed that $\cV$ is affine, the sheaf $\rho\lsta\bl\sO_{\cA}(\cA))$ is isomorphic to $\sO_\cV$.

\begin{lemm}\cite[Lemma 4.10]{HL10}
\label{lemm:usefulFacts}
We have
\begin{enumerate}
\item $\rho\lsta\sL^{\otimes r}\cong\sO_\cV\mathop{\oplus}\rho\lsta
\cO_{\cC_{\cV}}\left( \sum_{i=1}^{rd}\cD_i - \cB \right)$;
\item $\rho\lsta
\cO_{\cC_{\cV}}\left( \sum_{i=1}^{rd}\cD_i - \cB \right)\cong
\ker\varphi$;
\item $\mathop{\oplus}_{i=1}^{rd} \eta_i: \bigoplus_{i=1}^{rd}
\rho_*\sM_i\lra \rho_*\sM$ is an isomorphism, and $\mathop{\oplus}_{i=1}^{rd}
\varphi_i={\varphi \circ \mathop{\oplus}_{i=1}^{rd}\eta_i}.$
\end{enumerate}
\end{lemm}

Note that $\rho_* \sM_i \cong \cO_{\cV}$ and $\rho_*(\cO_{\cA}(\cA)) \cong \cO_{\cV}$ since we may assume $\cV$ sufficiently small. Then $\vphi_i$ is a morphism between trivial bundles. To describe each morphism $\vphi$ explicitly, we review arguments in \cite[Section 4]{HL10}.

For a weighted genus one nodal curve $C$, Let $\gamma^0$ be the associated dual graph. Then we contract a subgraph of $\gamma^0$ comes from the core subcurve, making the new graph $\gamma^1$. We denote the contracted vertex by `o'. o is also called the root of the graph. Using the following four operations on the rooted tree $\gamma^1$, {\it pruning}, {\it collapsing}, {\it specialization}, and {\it advancing}, we obtain a terminally weighted tree $\gamma$. See \cite[Section 3.2]{HL10} for details. Here, `terminally weighted` means weights are concentrated on the terminal(=maximal order) vertices. Note that the vertex set of every rooted tree has natural order having the root vertex as a minimal element.

Let $\gamma$ be the terminally weighted tree associated to $(C, p_1,\cdots,p_k,L)$. The weight is given by degrees of $L$ on each components of $C$. For each vertex $v\in \gamma$ we define
$$\zeta_v=\zeta_q\in \Gamma (\sO_\cV),
$$
where $q$ is the associated node of $v$, and $\Sigma_q=\{\zeta_q=0\}$ is the locus such that the node $q$ is not smoothed.
For any terminal vertex $i \in \text{Ver}(\gamma)^t$, we let
$$
\zeta_{[i,o]} =  \prod_{i \succeq v\succ o}\zeta_v.
$$
We have the following theorem,
\begin{theo}\label{main}\cite[Lemma 4.16]{HL10}
The direct image sheaf $\rho_* \sL^{\otimes r} $ is a direct sum of
$\sO_\cV^{\mathop{\oplus} (rd-\ell+1)}$ with the kernel sheaf of the
homomorphism \beq \label{simHom}
\mathop{\oplus}_{i=1}^{\ell}\varphi_i:\sO_\cV^{\mathop{\oplus} \ell}\lra \sO_\cV,
\quad \varphi_i=c_i \cdot \zeta_{[i,o]}, \quad c_i\in \CC \eeq
 where $\ell$ is the number of terminals vertices of $\gamma$.
\end{theo}

For a point in $Q_{1,k}(\PP^n,d)$, let $U$ be a small neighborhood of it. We pick a smooth chart $\mathcal{V}\rightarrow \fM_{1,k,d}^{div}$, which contains the image of $U\rightarrow \fM_{1,k,d}^{div}$. Let $\cU= \cV \times_{\fM_{1,k,d}^{div}} U$ and $\cE_{\cV}$ be the total space of the vector bundle $\rho_*\sL(\cA)^{\mathop{\oplus} n}$. Let $p:\cE_{\cV}\rightarrow \cV$ be the projection.
Then the tautological restriction homomorphism
$$\text{rest}: \rho_* \sL(\cA)^{\mathop{\oplus} n}\lra \rho_* (\sL(\cA)^{\mathop{\oplus} n}|_\cA)
$$
lifts to a section
\beq \label{sec1} F \in \Gamma
(\cE_{\cV}, p^*\rho_* (\sL(\cA)^{\mathop{\oplus} n}|_{\cA})).
\eeq

Then there is a canonical open immersion $\cU\to (F=0)\sub
\cE_{\cV}$. To a terminal vertex
$b \in \text{Ver}(\gamma)^t$, we associate $n$ coordinate functions $w_{b,1},\cdots,w_{b,n}\in\Gamma(\sO_{\cE_{\cV}})$.
We then  set
$$\Phi_\gamma=(\Phi_{\gamma,1},\cdots,\Phi_{\gamma,n}), \quad
\Phi_{\gamma,e} = \sum_{b\in \text{Ver}(\gamma)^t} \zeta_{[b,o]}w_{b,e}.
$$
Similar to Hu and Li's \cite[Theorem 2.19]{HL10}, we have the following theorem
\begin{theo}\label{Main9}
For a point in $Q_{1,1}(\PP^n,d)$, let $\gamma$ be the associated weighted tree, choosing $\cV$ as above and shrinking it if necessary and fix an isomorphism $ p^*\rho_*(\sL(\cA)^{\mathop{\oplus} n}|_{\cA})\cong \sO_{\cE_{\cV}}^{\mathop{\oplus} n}$. Then we can find regular functions over $\cE_{\cV}$, $w_{b,1},\cdots,w_{b,n}$, from coordinate functions of $\sO_{\cE_{\cV}}^{\mathop{\oplus} n}$ and node-smoothing parameter functions $\zeta_i$ such that
$$
F=(\Phi_{\gamma,1},\cdots,\Phi_{\gamma,n}).
$$
\end{theo}

When $k=1$, let $\gamma$ be a stable terminally weighted rooted trees of total weight $d$. We can easily check that $\gamma$ is a one path trees. Therefore $\gamma$ has only one terminal vertex, so that we have
$$
\Phi_{\gamma,e} = \zeta_1 w_{e}, \ \Phi_{\gamma} = (\zeta_1 w_1, \dots, \zeta_1 w_n)$$
where $\zeta_1$ is a node-smoothing parameter correspond to the unique terminal vertex of $\gamma$.
Let us denote $\zeta_1$ by $\zeta$. The local equation for $Q_{1,1}(\PP^n,d)$ can be easily described as the following.
\begin{coro}\label{Main1}
For a point in $Q_{1,1}(\PP^n,d)$, choosing $\cV$ as above and shrinking it if necessary and fixed $ p^*\rho_*(\sL(\cA)^{\mathop{\oplus} n}|_{\cA})\cong \sO_{\cE_{\cV}}^{\mathop{\oplus} n}$, we can find $n+1$ regular functions $w_1,\cdots,w_{n},\zeta$ over $\cE_{\cV}$ such that
$$
F=(w_1\zeta,\cdots,w_{n}\zeta).
$$
Furthermore, each $w_i$ and $\zeta$ has smooth vanishing locus, which intersect transversally to each others.
\end{coro}

When $k>1$, as in \cite{HL10}, let $\Theta_s$ be the closure in $\fM^{wt}_{1,k}$ of the locus where the weight is zero on the genus one core component, and has $s$ rational components attach to the genus one component.
Let $\widetilde{\fM}^{wt}_{1,k}$ be the successive blow up $\fM^{wt}_{1,k}$ along $\Theta_1, \dots, \Theta_d $. Then irreducible components of $\widetilde{Q}_{1,k}(\PP,d):=Q_{1,k}(\PP,d)_{\fM^{wt}_{1,k}}\times\widetilde{\fM}^{wt}_{1,k}$ are smooth and intersect transversally, and we also have the following local equations. The following is a direct analogue of \cite[Theorem 2.19]{HL10} and \cite[Proposition 2.1]{LO} in stable quasi-map spaces.

\begin{theo}\label{Main10}
For a point in $\widetilde{Q}_{1,k}(\PP,d)$ choosing an smooth affine chart $\widetilde{\cV}$ of $ \widetilde{\fM}_{1,k} $, and shrinking it if necessary and fixed $ p^*\rho_*(\sL(\cA)^{\mathop{\oplus} n}|_{\cA})\cong \sO_{\cE_{\widetilde{\cV}}}^{\mathop{\oplus} n}$, we can find $n + d'$ regular functions $w_1,\cdots,w_{n}$ and $\zeta_1,\dots,\zeta_{d'}$ over $\cE_{\widetilde{\cV}}$ where $d'=\min\{k,d\}$, such that
$$
F=(w_1\tau, \cdots, w_{n}\tau ), \quad \tau := \zeta_1\cdots \zeta_{d'}.
$$
Furthermore, each $w_i$ and $\zeta_j$ has smooth vanishing locus, and they intersect transversally to each others.
\end{theo}

Set $\cX=Q_{1,1}(\PP^n,d)$, let $\pi_{\cX}:\cC_{\cX}\to \cX$ be the universal family and $\sL_{\cX}$ be the universal line bundle over $\cC_{\cX}$.  By the stability conditions, we know that $\cX$ has two different irreducible components, the main component $\cX_{\rede}$ (where the underlying curves of the generic points are smooth elliptic curves),  and the other is the so called ghost component $\cX_{\gst}$.
Locally, $\cX_{\rede} = \{w_1 = \dots = w_n = 0 \}$ and $\cX_{\gst} = \{ \tau = 0 \}$.
Then by the proof of \cite[Theorem 2.11]{HL10}, we have
\begin{theo}\label{Main3}
The direct image sheaf $\pi_{\cX_{\rede *} } \left( \sL_{\cX}^{\otimes r} |_{\cX_{\rede} } \right)$ is locally free of rank $rd$,
and the direct image sheaf $\pi_{\cX_{\gst *} } \left( \sL_{\cX}^{\otimes r} |_{\cX_{\gst} } \right)$ is locally free of rank $rd + 1$.
\end{theo}

\begin{rema}
For $k>1$, we can obtain similar result as Theorem \ref{Main3}. In this case, ghost component is not irreducible. For each irreducible component of $\wtil{Q}_{1,k}(\PP^n,d)$, denoted by $\wtil{Q}_\gamma$, the direct image sheaves $\pi_{\wtil{Q}_{\gamma} * } \left( \sL_{\wtil{Q}_{1,k}(\PP^n,d)}^{\otimes r} |_{\wtil{Q}_\gamma } \right)$ is locally free of rank $rd+1$.
Also, the direct image sheaf $\pi_{\wtil{Q}_{\rede} *} \left( \sL_{\wtil{Q}_{1,k}(\PP^n,d)}^{\otimes r} |_{\wtil{Q}_{\rede} } \right)$ is locally free of rank $rd$, where $\wtil{Q}_{\rede}$ denotes the main component.
\end{rema}

\section{Moduli of stable quasimaps with fields}
\def\upf{^{\oplus 5}}
\def\umtf{{-\otimes 5}}
\subsection{Stable quasimaps with fields}
First we recall the moduli stack of stable quasimaps with fields introduced in \cite{L}.
To simplify the notation, we will focus on the genus one case. Let us abbreviate $Q:=Q_{1,k}(\PP^n,d)$. Let
$$\pi_{Q}: \cC_{Q}\lra Q,\quad
\sP^i_{Q}=\sL_{Q}^{\vee\otimes \deg q_i}\otimes \omega_{\cC_{Q}/Q },\quad 1\le i\le m.
$$
As in \cite{L}, let $\cP = \cP_{1,k}=C(\oplus_{i=1}^m\pi_{Q*}\sP^i_Q)$ be the cone stack over $Q$.
The relative perfect obstruction theory over
$\cP\to \fM^{line}_{1,k}$ is given by
\beq\label{y-o}
\phi_{\cP/\fM^{line}_{1,k}}:\TT_{\cY/\fM^{line}_{1,k}} 
\lra  \mathbb{E}\dual_{\cP/\fM^{line}_{1,k}} ,\quad \mathbb{E}\dual_{\cP/\fM^{line}_{1,k}}\defeq   R^\bullet \pi_{\cP\ast}(\sL_\cP^{\oplus (n+1)}\bigoplus \oplus_i\sP^i_\cP),
\eeq
where $$\pi_{\cP}: \cC_\cP\lra \cP,\quad
\sP^i_\cP=\sL_\cP^{\vee\otimes \deg q_i}\otimes \omega_{\cC_\cP/\cP},\quad 1\le i\le m
$$ is the universal curve and $\TT_{\cP/\fM^{line}_{1,k}}$ denotes the relative tangent complex.


According to the convention, we call the cohomology sheaf
$$\Ob_{\cP/\fM^{line}_{1,k}}\defeq  H^1(\mathbb{E}\dual_{\cP/\fM^{line}_{1,k}})=R^1 \pi_{\cP\ast}(\sL_\cP^{\oplus (n+1)}\bigoplus \oplus_i\sP^i_\cP)
$$
the relative obstruction sheaf of $\phi_{\cP/\fM^{line}_{1,k}}$.


\vsp

The authors \cite{L} constructed a cosection of $\Ob_{\cP/\fM^{line}_{1,k}}$ by using the defining polynomials $q_1(x)=\cdots=q_m(x)=0$ of $X$. Namely a homomorphism
\beq\label{cos}
\sigma': \Ob_{\cP/\fM^{line}_{1,k}}\lra \sO_\cP.
\eeq
 This cosection
can be lifted to a cosection $\widetilde{\sigma}': \Ob_\cP\to\sO_\cP$
of the obstruction sheaf $\Ob_\cP$. Note that the obstruction sheaf $\Ob_\cP$ fits into the exact sequence
$$\mathfrak{f}_{\cP}\sta\TT_{\fM^{line}_{1,k}}\lra \Ob_{\cP/\fM^{line}_{1,k}}\lra \Ob_\cP\lra 0.
$$
The degeneracy locus $D(\sigma')$ of $\sigma'$,  where
$\sigma$ is not surjective, is the closed subset
\beq\label{emb}
D(\si')=Q_{1,k}(X,d)\sub \cP.
\eeq
Moreover we have $A_{*} D(\si')=A_{*} Q_{1,k}(X,d)$ by the result in \cite{CL}.
Furthermore, in \cite{CL} the authors defined the (localized) virtual cycle for $\cP$ as
\[
[\cP]\virt_\loc := 0^!_{\sigma',\loc}[\fC_{\cP / \fM_{1,k}^{line}}] \in A_{*} D(\si')=A_{*} Q_{1,k}(X,d)
\]
where $0^!_{\sigma',\loc}$ is the the localized Gysin map defined in \cite{KL} for the cosection $\sigma'$, and $\fC_{\cP/\fM^{line}_{1,k}}$ is the relative intrinsic normal cone.


\begin{theo}[\cite{L},\cite{KimJ}] \label{theorem} We have
$$[\cP]_{\loc}\virt=(-1)^{(\sum_i\deg q_i)d}[Q_{1,k}(X,d))]^{\virt}.
$$
\end{theo}

We remark that this Theorem holds for all genus $g$ and $k$. For our purpose here, we only state in the case $g=1$.
Set $\phi:\fM^{line}_{1,k}\to  \fM_{1,k}$. Then we have the following distinguished triangles
\beq
\mathfrak{f}_{\cP}^*\mathbb{T}_{\fM^{line}_{1,k}}[-1]\lra \mathbb{T}_{\cP/\fM^{line}_{1,k}}\lra \mathbb{T}_{\cP}\lra \mathfrak{f}_{\cP}^*\mathbb{T}_{\fM^{line}_{1,k}} .
\eeq

By \cite[Lemma 3.6]{CL}, the composing with $\sigma'\circ H^1(\phi_{\cP/\fM^{line}_{1,k}}): \mathbb{T}_{\cP/\fM^{line}_{1,k}}\lra \mathbb{E}\dual_{\cP/\fM^{line}_{1,k}}\lra \sO_\cP$ is zero.  From the following commutative diagram, the cosection $\sigma'$ induces a cosection
 $ \sigma:H^1(\mathbb{E}_{\cP/\fM_{1,k}})\to \sO_\cP$.

\beq
\begin{CD}
H^1(\mathfrak{f}_{\cP}^*\mathbb{T}_{\fM^{line}_{1,k}/\fM_{1,k}}[-1])@>^{=}>>H^1(\mathfrak{f}_{\cP}^*\mathbb{T}_{\fM^{line}_{1,k}/\fM_{1,k}}[-1])
\\
@VVV @VVV \\
H^1(\mathbb{T}_{\cP/\fM^{line}_{1,k}})@>^{H^1(\phi_{\cP/\fM^{line}_{1,k}})}>>H^1(\mathbb{E}\dual_{\cP/\fM^{line}_{1,k}})@>^{\sigma'}>>\sO_\cP
\\
@V^{\theta_{int}} VV @V^\theta VV@V^{=}VV \\
H^1(\mathbb{T}_{\cP/\fM_{1,k}})@>^{H^1(\phi_{\cP/\fM_{1,k}})}>>H^1(\mathbb{E}\dual_{\cP/\fM_{1,k}})@>^{\sigma}>>\sO_\cP
.
\end{CD}
\eeq
By \cite[Proposition 3.5]{CL}, the following morphism
\beq
\eta: H^1(\mathfrak{f}_{\cP}^*\mathbb{T}_{\fM^{line}_{1,k}}[-1])\lra H^1(\mathbb{T}_{\cP/\fM^{line}_{1,k}})\lra H^1(\mathbb{E}\dual_{\cP/\fM^{line}_{1,k}})\lra \sO_\cP
\eeq
is zero. Let $\mathfrak{g}_{\cP}:=\phi\circ\mathfrak{f}_{\cP}:\cP\to \fM_{1,k}$. By the commutative diagram below
\beq  \setlength{\unitlength}{0.5mm}
\begin{CD}
H^1(\mathfrak{f}_{\cP}^*\mathbb{T}_{\fM^{line}_{1,k}/\fM_{1,k}}[-1])@>^{=}>>H^1(\mathfrak{f}_{\cP}^*\mathbb{T}_{\fM^{line}_{1,k}/\fM_{1,k}}[-1])@>^{=}>>H^1(\mathfrak{f}_{\cP}^*\mathbb{T}_{\fM^{line}_{1,k}/\fM_{1,k}}[-1])
\\
@VVV@VVV @VVV \\
H^1(\mathfrak{f}_{\cP}^*\mathbb{T}_{\fM^{line}_{1,k}}[-1])@>>>H^1(\mathbb{T}_{\cP/\fM^{line}_{1,k}})@>^{H^1(\phi_{\cP/\fM^{line}_{1,k}})}>>H^1(\mathbb{E}\dual_{\cP/\fM^{line}_{1,k}})@>^{\sigma'}>>\sO_\cP
\\
@VVV@V^{\theta_{int}} VV @V^\theta VV@V^{=}VV \\
H^1(\mathfrak{g}_{\cP}^*\mathbb{T}_{\fM_{1,k}}[-1])@>>>H^1(\mathbb{T}_{\cP/\fM_{1,k}})@>^{H^1(\phi_{\cP/\fM_{1,k}})}>>H^1(\mathbb{E}\dual_{\cP/\fM_{1,k}})@>^{\sigma}>>\sO_\cP
,
\end{CD}
\eeq
the morphism
\beq
\eta': H^1(\mathfrak{g}_{\cP}^*\mathbb{T}_{\fM_{1,k}}[-1])\lra H^1(\mathbb{T}_{\cP/\fM_{1,k}})\lra H^1(\mathbb{E}\dual_{\cP/\fM_{1,k}})\lra \sO_\cP
\eeq
obtained by the composition is zero. Thus the cosection $\sigma: H^1(\mathbb{E}\dual_{\cP/\fM_{1,k}})\lra \sO_\cP$ can be lifted to the cosection $\Ob_\cP\to \sO_\cP$. Therefore we can define the following virtual cycle
\beq
0^!_{\si,\loc}[\mathfrak{C}_{\cP/\fM_{1,k}}].
\eeq

Since $\phi:\fM^{line}_{1,k}\to  \fM_{1,k}$ is smooth, we have the following commutative diagram:
\beq
\begin{CD}
h^1/h^0(R^{\bullet}\pi_{\cP}\cO_{\cP})@>^f>>h^1/h^0(\mathbb{E}\dual_{\cP/\fM^{line}_{1,k}})@>^\theta>> h^1/h^0(\mathbb{E}\dual_{\cP/\fM_{1,k}})
\\
@AA^{=} A@A^{\phi_{\cY/\fM^{line}_{1,1}}}AA @A^{\phi_{\cY/\fM_{1,1}}}AA \\
h^1/h^0(R^{\bullet}\pi_{\cP}\cO_{\cP})@>>>h^1/h^0(\mathbb{T}_{\cP/\fM^{line}_{1,k}})@>^{\theta_{int}}>> h^1/h^0(\mathbb{T}_{\cP/\fM_{1,k}})
.
\end{CD}
\eeq
Hence we have
\beq
\theta^*\mathfrak{C}_{\cP/\fM_{1,k}}=\mathfrak{C}_{\cP/\fM^{line}_{1,k}}.
\eeq
Therefore
\beq
0^!_{\si,\loc}[\mathfrak{C}_{\cP/\fM_{1,k}}]=0^!_{\si',\loc}[\mathfrak{C}_{\cP/\fM^{line}_{1,k}}] = [\cP]\virt_\loc.
\eeq




When $k=1$, let $\cX=Q_{1,1}(\PP^n,d)$ as in Section \ref{sect:local}. Let
$$\pi_{\cX}: \cC_\cX\lra \cX,\quad
\sP^i_\cX=\sL_\cX^{\vee\otimes \deg q_i}\otimes \omega_{\cC_\cX/\cX},\quad 1\le i\le m.
$$
Let us denote $\cP_{1,1}$ by $\cY$, which is a cone over $\cX$, and let $\mathfrak{f}_{\cY}:\cY\to \fM_{1,1}$ be the forgetful morphism. For any closed point $y=[(C,p_1,L,\{u_i\}^{n}_{i=0})]\in \cY$, let $\cV \to \fM_{1,1}^{div}$ be a smooth affine chart. Since the forgetful morphism $\fM_{1,1,d}^{div}\to \fM_{1,1}$ is smooth, 
$\cV$ is also a smooth affine chart of $\fM_{1,1}$. We may assume that $[( (\cC_{\cV})_0, p_1(0))] = [ (C, p_1 ) ] = \ff_{\cY}(y)$.
Here, $\cC_{\cV}$ is a canonical curve over $\cV$ and $p_1 : \cV \to \cC_{\cV}$ are universal sections. Then by \cite[Proposition 3.1]{CL3} and its proof, we have
\begin{prop}\label{coordinate}
Let $\cU$ be a small affine chart of the closed $y\in \cY$, and $\ff_{\cY}(\cU)\subset \cV$, then $\cU$ can be open embedded in the substack $F^{-1}(0)$, where
\begin{align*}
F:\cV\times  \CC^{dn}\times \CC^{n+m} & \to  \cO_{\cV\times \CC^{dn} \times \CC^{n+m}}^{\oplus (n+m)} \oplus \cO_{\cV\times \CC^{dn} \times \CC^{n+m}}^{\oplus d n} \\
(z,w_1\cdots,w_{n},t) & \mapsto (w_1 \zeta,\cdots,w_n\zeta, t \zeta ,0,\dots,0),\nonumber
\end{align*}
and $z\in \cV\times \CC^{dn}$, $\zeta$ is a regular function on $\cV$, $w_i$ are coordinates of $\CC^{n+1}$ and $t=(t_1,\cdots,t_m)$ are coordinates of $\CC^m$.
\end{prop}

\begin{rema}\label{div-loc}
The section $F$ gives the local model of the moduli space $\cY$ over $\fM_{1,1}^{div}$. It means that the differential of the section $F$
\[
\TT_{\cU / \cV } = \cO_{\cU}^{\oplus dn} \oplus \cO_{\cV\times \CC^{dn} \times \CC^{n+m}}^{\oplus (n+m)} \stackrel{dF}{\longrightarrow} \cO_{\cU}^{\oplus n+m} \oplus \cO_{\cU}^{\oplus d n}
\]
coincides with the (dual) relative perfect obstruction theory $\EE_{\cU / \fM_{1,1}^{div}}$.
\end{rema}
Since locally $\cY=\{w_1=\cdots=w_n=t_1=\cdots=t_m=0\} \cup  \{\zeta_1=0\} $, we know $\cY$ has two different irreducible components $\cY_{\rede}$ and $\cY_\gst$ with $\cY_{\rede}=\cX_{\rede}$.

%
%



\subsection{Comparison of relative perfect obstruction theories}

The natural relative perfect obstruction theory $\EE_{\cY / \fM^{line}_{1,1} }$ is defined as the following
\begin{align}\label{eq:relPOT4}
\EE\dual_{\cY /\fM^{line}_{1,1} } =
p^*\EE\dual_{\cX / \fM^{line}_{1,1} } \bigoplus  \oplus_ip^* R^{\bullet}\pi_{\cX*}(\sP^i_\cX)
\end{align}
where $p : \cY \to \cX$ is the forgetful morphism, $\pi_{\cX}: \cC_\cX \to \cX$ is the universal curve, and $\sL_{\cX}$ is the universal line bundle over $\cC_\cX$. Thus
\begin{align}\label{eq:relPOT4-1}
H^1(\EE\dual_{\cY /\fM^{line}_{1,1} }) =
p^*H^1(\EE\dual_{\cX / \fM^{line}_{1,1} }) \bigoplus  \oplus_ip^* H^1(R^{\bullet}\pi_{\cX*}(\sP^i_\cX))
\end{align}

From the cotangent complexes associated to the triples $\cY \to \fM_{1,1}^{line} \to \fM_{1,1}$ and $\cX \to \fM_{1,1}^{line} \to \fM_{1,1}$, we obtain the diagram
\begin{align*}
\xymatrix{
H^1(\mathfrak{f}_{\cY}^*\mathbb{T}_{\fM^{line}_{1,1}/\fM_{1,1}}[-1]) \ar[r]^-{=} \ar[d] & H^1(\mathfrak{f}_{\cY}^*\mathbb{T}_{\fM^{line}_{1,1}/\fM_{1,1}}[-1]) \ar[d] \\
H^1(\mathbb{T}_{\cY/\fM^{line}_{1,1}}) \ar[r] \ar[d] & p^*H^1(\mathbb{T}_{\cX/\fM^{line}_{1,1}}) \ar[d] \\
H^1(\mathbb{T}_{\cY/\fM_{1,1}}) \ar[r] & p^*H^1(\mathbb{T}_{\cX/\fM_{1,1}}).
}
\end{align*}
Furthermore, we have
\begin{align*}
\xymatrix{
H^1(\mathfrak{f}_{\cY}^*\mathbb{T}_{\fM^{line}_{1,1}/\fM_{1,1}}[-1]) \ar[r]^-{=} \ar[d] & H^1(\mathfrak{f}_{\cY}^*\mathbb{T}_{\fM^{line}_{1,1}/\fM_{1,1} }[-1]) \ar[d] \\
H^1(\mathbb{E}\dual_{\cY/\fM^{line}_{1,1}}) \ar[r] \ar@{->>}[d] & p^*H^1(\mathbb{E}\dual_{\cX/\fM^{line}_{1,1}}) \ar@{->>}[d] \ar@/_1.5pc/[l]_-{j} \\
H^1(\mathbb{E}\dual_{\cY/\fM_{1,1}}) \ar[r] & p^*H^1(\mathbb{E}\dual_{\cX/\fM_{1,1}}) \ar@/_1.5pc/[l]_-{\bar{j}}.
}
\end{align*}
Here $j$ is a morphism which gives the splitting \eqref{eq:relPOT4-1} of $H^1(\EE\dual_{\cY/\fM^{line}_{1,1}})$. Note that the vertical arrows $H^1(\EE\dual_{\cY / \fM^{line}_{1,1}}) \to H^1(\EE\dual_{\cY / \fM_{1,1}})$ and $p^*H^1(\EE\dual_{\cX / \fM^{line}_{1,1}}) \to p^*H^1(\EE_{\cX / \fM_{1,1}})$ are surjective. Then, by chasing the diagram we can show that $j$ induce the morphism $\bar{j}$, which gives the splitting. So we obtain the decomposition
\begin{align}\label{eq:relPOT4-2}
H^1(\EE\dual_{\cY /\fM_{1,1} }) =
p^*H^1(\EE\dual_{\cX / \fM_{1,1} }) \bigoplus  \oplus_i \ p^* H^1(R^{\bullet}\pi_{\cX*}(\sP^i_\cX)).
\end{align}

Parallel to \cite[Lemma 2.4]{LO2}, we will prove the following lemma.

\begin{lemm}
(1) For a sufficiently small open neighbourhood $\cU \subset \cX$, and $\cU\lgst : = \cU \times_{\cX} \cX\lgst$ we have
\begin{align}\label{eq:obsbundleisom1}
H^1( \EE\dual_{\cU /\fM_{1,1}^{div} }|_{\cU\lgst} ) \stackrel{\cong}{\longrightarrow} H^1( \EE\dual_{\cU /\fM_{1,1} }|_{\cU\lgst} ).
\end{align}

(2) Also, for a sufficiently small open neighbourhood $\cU' \subset \cY$, and $(\cU')\lgst : = \cU' \times_{\cX} \cX\lgst$ we have
\begin{align}\label{eq:obsbundleisom2}
H^1( \EE\dual_{\cU' /\fM_{1,1}^{div} }|_{ (\cU')\lgst} ) \stackrel{\cong}{\longrightarrow} H^1( \EE\dual_{\cU' /\fM_{1,1} }|_{ (\cU')\lgst} ).
\end{align}

\end{lemm}

\begin{proof} Since the proof of (2) is parallel to (1), we will only prove (1) here. We first consider the neighbourhood $\cU \subset \cX$. We may assume that $\cU \subset \cU_0$.
Note that $\EE\dual_{\cX / \fM^{line}_{1,1}}|_{\cU} \simeq R^\bullet\pi_{\cX*}\cO_{\cC_{\cU}}(\cD_{\cU})^{\oplus n+1}$ on the neighbourhood $\cU$. Recall the remark \ref{distingtriRemark}, which says that the horizontal arrow
 $\phi : R^\bullet\pi_{\cX*} \cO_{\cC_{\cX}}|_{\cU} \to \EE\dual_{\cX / \fM^{line}_{1,1}} |_{\cU}$ in \eqref{eq:POTdiag1} is induced from the arrow
\[
\xymatrix@C=35pt{
\cO_{\cC_{\cU}} \ar[r]_-{(s_0,\dots,s_n)} & \cO_{\cC_{\cU}}(\cD_{\cU})^{\oplus n+1}
}
\]
by taking $R^\bullet \pi_{\cX*}(-)$.
Consider the exact sequence of complexes
\begin{align}\label{eq:disting1}
0 \to [0 \to \cO_{\cC_{\cU}}(\cD_{\cU})^{\oplus n}] \to [\xymatrix@C=35pt{
\cO_{\cC_{\cU}} \ar[r]_-{(s_0,\dots,s_n)} & \cO_{\cC_{\cU}}(\cD_{\cU})^{\oplus n+1}
}] \to [\cO_{\cC_{\cU}} \stackrel{s_0}{\lra} \cO_{\cC_{\cU}}(\cD_{\cU})] \to 0.
\end{align}
Since $\EE\dual_{\cX/ \fM_{1,1}} |_{\cU} $ is equivalent to the mapping cone $cone(\phi)$, and $[\cO_{\cC_{\cU}} \stackrel{s_0}{\lra} \cO_{\cC_{\cU}}(\cD_{\cU})] \stackrel{qis}{\simeq} \cO_{\cD_{\cU}}$, we have the distinguished triangle
\begin{align*}
\EE\dual_{\cX / \fM^{div}_{1,1}} |_{\cU} \to \EE\dual_{\cX / \fM_{1,1}} |_{\cU} \to R^\bullet\pi_{\cX*}\cO_{\cD_{\cU}} \stackrel{+1}{\lra}
\end{align*}
by taking $R^\bullet\pi_{\cX*}$ to the sequence \eqref{eq:disting1}. Then, by taking the long exact sequence of this distinguished triangle, we obtain the exact sequence:
\begin{align}\label{eq:Obscompare1}
H^1(\EE\dual_{\cX / \fM^{div}_{1,1} }|_x) \to H^1(\EE\dual_{\cX / \fM_{1,1}} |_x) \to 0
\end{align}
for any closed point $x \in \cU$.

\medskip

On the other hand, we can consider the short exact sequence
\[
0 \to \cO_{\cC_{\cU}} \stackrel{s_0}{\lra} \cO_{\cC_{\cU}}(\cD) \to \cO_{\cD} \to 0
\]
where $\cD = s_0^{-1}(0)$ is the family of degree $d$ divisors on the universal curve $\cC_{\cU} \to \cU$.
Therefore we have the short exact sequence
\[
0 \to \cO_{\cC_{\cU}} \stackrel{s_0}{\lra} \cO_{\cC_{\cU}}(\cD)^{\oplus n+1} \to \cO_{\cC_{\cU}}(\cD)^{\oplus n} \oplus \cO_{\cD} \to 0.
\]
From the isomorphisms
\[
\EE\dual_{\cX / \fM_{1,1}}|_{\cU} \simeq
\mathrm{Cone}[ R^\bullet\pi_{\cX*}\cO_{\cC_{\cU}} \to  \EE\dual_{\cX / \fM^{line}_{1,1}}|_{\cU}], \textrm{ and }
\EE\dual_{\cX / \fM^{line}_{1,1}}|_{\cU} \simeq R^\bullet\pi_{\cX*} \cO_{\cC_{\cU}}(\cD)^{\oplus n+1}
\]
we obtain
\begin{align}\label{eq:Cohdims}
\dim H^0(\EE\dual_{\cX / \fM_{1,1}}|_x ) & = h^0(C_x, \cO_{C_x}(D_x))^{\oplus n} + h^0(C_x, \cO_{D_x}) = n \cdot h^0(C_x, \cO_{C_x}(D_x)) + d,
\\ \nonumber
\dim H^1(\EE\dual_{\cX / \fM_{1,1}}|_x ) & = h^1(C_x, \cO_{C_x}(D_x))^{\oplus n} + h^1(C_x, \cO_{D_x}) = n \cdot h^1(C_x, \cO_{C_x}(D_x))
\end{align}
for each closed point $x \in \cU$. The fiber $C_x = \cC_{\cU}|_x$ of the universal curve over $x$ and the degree $d$ divisor $D_x = \cD|_x$ on $C_x$, which is the fiber of the universal divisor $\cD$ over $x$.
If $x \in \cU\lgst $, we observe that
\[
\dim H^0(\EE\dual_{\cX / \fM_{1,1}}|_x ) = n(d+1) + d, \ \dim H^1(\EE\dual_{\cX / \fM_{1,1}}|_x ) = n
\]
from \eqref{eq:Cohdims}. Also it is trivial that $\dim H^1(\EE\dual_{\cX / \fM^{div}_{1,1}}|_x ) = n \cdot h^1(C_x, \cO_{C_x}(D_x)) = 1$ for $x \in \cU\lgst$.
Therefore, for an arbitrary closed points $x \in \cU\lgst$, the morphism
\[
H^1(\EE\dual_{\cX / \fM^{div}_{1,1}} |_x) \to H^1(\EE\dual_{\cX / \fM_{1,1}} |_x)
\]
from \eqref{eq:Obscompare1} is an isomorphism since it is surjective and both vector spaces have same dimension $n$.
Since $\cU\lgst$ is a reduced scheme, we have an isomorphism
\begin{align*}
H^1(\EE\dual_{\cX / \fM^{div}_{1,1}} |_{\cU\lgst} ) \stackrel{\cong}{\lra} H^1(\EE\dual_{\cX / \fM_{1,1}} |_{\cU\lgst} )
\end{align*}

\end{proof}
Because the sheaf $H^1(\EE\dual_{\cX / \fM^{div}_{1,1}} |_{\cU\lgst} )$ is locally free by Remark \ref{div-loc}, we have the following.

\begin{prop} \label{loc}The obstruction sheaf $H^1(\EE\dual_{\cX / \fM_{1,1}} |_{\cU\lgst} )$ is locally free.
\end{prop}


\subsection{Decomposition of the intrinsic normal cone}

Let $R=\text{Spec}(B)$ be a smooth affine variety. Let $\tilde{R}:=R\times \CC^{n+m}$, and $F$ be the section of $\cO_{\tilde{R}}^{n+m}$ with $F=(w_1\zeta,\cdots,w_{n+m}\zeta)$, where $w_i$ are coordinates of $\CC^{n+m}$, and $\zeta \in B$ is a regular function. Denote by $\cZ=F^{-1}(0)$ the zero loci of $F$. Then $\cZ$ has two different components, where $\cZ=\cZ_1\cup\cZ_2$ with $\cZ_1=\{w_1=\cdots=w_{n+m}=0\}$ and $\cZ_2=\{\zeta=0\}$.

\begin{lemm}\label{cone0}
Let $C_{\cZ/\tilde{R}}$ be the normal cone of $\cZ$ in $ \tilde{R}$, then $C_{\cZ/\tilde{R}}=C_1\cup C_2$ has two different irreducible components $C_1$ and $C_2$ support on $\cZ_1$ and $\cZ_2$ respectively, and there is a canonical dominant morphism
\beq
C_{\cZ_2/\tilde{R}}\to C_{\cZ/\tilde{R}}|_{\cZ_2}.
\eeq

\end{lemm}
\begin{proof}Let $\fR:=B[w_1,\cdots,w_{n+m}]/(w_1\varsigma_1,\cdots,w_{n+m}\varsigma_1)$ be the coordinate ring of $\cZ$. Consider the following surjective morphism
\beq
\fR[A_1\cdots,A_{n+m}]\mapsto \bigoplus_{k\ge 0} I^k_{\cZ/\tilde{R}}/I^{k+1}_{\cZ/\tilde{R}}
\eeq
$$
A_i\mapsto w_i\varsigma_1.
$$
Then $C_{\cZ/\tilde{R}}=\text{Spec}\bigg(\fR[A_1\cdots,A_{n+m}]/(w_iA_j-w_jA_i)\bigg)$, which supports on $\cZ_1$ and $\cZ_2$. We have
\begin{eqnarray*}
C_{\cZ/\tilde{R}}|_{\cZ_1}&=&\text{Spec}\bigg(\fR[A_1\cdots,A_{n+m}]/(w_iA_j-w_jA_i)\otimes \fR/(w_1,\cdots,w_{n+m}) \bigg) \\ \nonumber
&=&\text{Spec}(B[A_1\cdots,A_{n+m}]),
\end{eqnarray*}
and

\begin{align}\label{eq:localcone2}
C_{\cZ/\tilde{R}}|_{\cZ_2}&=&\text{Spec}\bigg(\fR[A_1\cdots,A_{n+m}]/(w_iA_j-w_jA_i)\otimes B[w_1,\cdots,w_{n+m}]/(\varsigma_1) \bigg)\\ \nonumber
&=&\text{Spec}\bigg(B/(\varsigma_1)[w_1,\cdots,w_{n+m}][A_1\cdots,A_{n+m}]/(w_iA_j-w_jA_i)\bigg),
\end{align}
Thus the fiber over $C_{\cZ/\tilde{R}}|_{\cZ_2}$ over $\cZ_2$ is the affine cone of the blowing up $\mbox{Bl}_0\CC^{n+m}$, and $C_{\cZ/\tilde{R}}|_{\cZ_1}$ is a vector bundle over $\cZ_1$. They are all irreducible. Hence $C_{\cZ/\tilde{R}}|_{\cZ_2}$ and $C_{\cZ/\tilde{R}}|_{\cZ_1}$ are irreducible.

Because $\cZ_2\subset\cZ\subset \tilde{R}$, there is a canonical morphism
\beq \label{eq:localmorph}
C_{\cZ_2/\tilde{R}}\to C_{\cZ/\tilde{R}}|_{\cZ_2}.
\eeq

The ideal $I_{\cZ_2/\tilde{R}}$ is equal to $(\zeta)$, the cone $C_{\cZ_2 / \wtil{R}}$ is isomorphic to $N_{\cZ_2/\tilde{R}}$ which is a line bundle. Since $I_{\cZ / \wtil{R}} = (w_1 \zeta, \dots, w_{n+m} \zeta)$, the composition of the morphism \eqref{eq:localmorph} with the inclusion $C_{\cZ/\wtil{R}}|_{\cZ_2} \hookrightarrow \cZ_2 \times \CC^{n+m}$ is given by
\begin{align}\label{eq:localmorph2}
N_{\cZ_2/\tilde{R}} = C_{\cZ_2/\tilde{R}} & \to C_{\cZ/\tilde{R}}|_{\cZ_2} \hookrightarrow \cZ_2 \times \CC^{n+m} \\ \nonumber
1 & \mapsto (w_1,\dots, w_{n+m})
\end{align}
where $1$ is a local generator of the line bundle. From the local description \eqref{eq:localcone2} of $C_{\cZ/\wtil{R}}|_{\cZ_2}$, we can check that $C_{\cZ/\wtil{R}}|_{\cZ_2}$ is a closure of the image of the above morphism. Hence \eqref{eq:localmorph} is dominant.
\end{proof}

%

%

Let $\cV$ and $\cU$ be smooth affine charts of $\fM_{1,1}$ and $ \cY$ as in Proposition \ref{coordinate}. Denote by $\widetilde{\cU}:=\cV\times \CC^{dn}\times \CC^{n+m}$.  Then the cone $\mathfrak{C}_{\cY/\fM_{1,1}}|_{\cU}=[C_{\cU/\widetilde{\cU}}/T_{\widetilde{\cU}|_{\cU}}]$ has two different components by Proposition \ref{coordinate} and Lemma \ref{cone0}. Hence $\mathfrak{C}_{\cY/\fM_{1,1}}$ has two different components. Denote them by
\beq\mathfrak{C}_{\cY/\fM_{1,1}}=\mathfrak{C}_{\rede}\cup\mathfrak{C}_{\gst},\eeq
which are supported on $\cY_{\rede}$ and $\cY_{\gst}$ respectively. Consequently,
\beq\label{split-0}
[\cY]\virt\lloc=0_{\si,\loc}^![\mathfrak{C}_{\rede}]+0_{\si,\loc}^![\mathfrak{C}_{\gst}].
\eeq
Let $C_{\gst}$ be the coarse moduli space of $\mathfrak{C}_{\gst}$, then $C_{\gst}\subset H^1(\mathbb{E}\dual_{\cY/\fM_{1,1}})|_{\cY_{\gst}}$.


Let us define $\fM_{\gst}:= \iota\left( \overline{M}_{1,1}\times \fM_{0,2} \right) \subset \fM_{1,1}$ 
where $\iota$ is the node-identifying morphism.
It is a substack whose general points are stable genus one curves attached by rational tails.
Moreover let $\mathfrak{g}_{\cY} : \cY \to \fM_{1,1}$ be the forgetful morphism and $\mathfrak{g}_{\cY_{\gst}}: \cY_{\gst}\to \fM_{1,1}$ be the restriction of $\mathfrak{g}_{\cY}$ on $\cY_{\gst}$. Consider the coarse moduli space $C_{\cY_{\gst}/\fM_{1,1}}$ of the intrinsic normal cone $\fC_{\cY_{\gst}/\fM_{1,1}}$. Note that we have
\beq \label{eq:normalbundle}
C_{\cY_{\gst}/\fM_{1,1}} = \mathfrak{g}_{\cY_{\gst}}^*N_{\fM_{\gst}/\fM_{1,1}},
\eeq
where $N_{\fM_{\gst}/\fM_{1,1}}$ is the normal bundle of $\fM_{\gst}\subset \fM_{1,1}$. Since $\cY_{\gst}\subset \cY$, there is a nature morphism
\beq \label{eq:naturalmorphism}
\phi: C_{\cY_{\gst}/\fM_{1,1}}\to C_{\cY / \fM_{1,1}}|_{\cY\lgst}= C_{\gst} \subset
H^1(\mathbb{E}\dual_{\cY/\fM_{1,1}} |_{\cY_{\gst}} ).
\eeq

By \eqref{eq:localmorph2}, $\phi$ is locally expressed by
\beq \label{localmorph}
\phi|_{\cU}:1\mapsto (w_1,\cdots,w_{n+m},0,\cdots,0).
\eeq
Moreover, from the above local computation for the normal cone, we observe that $\phi$ is a birational morphism. Hence $C_{\gst}$ is birational to the line bundle $\mathfrak{g}_{\cY_{\gst}}^*N_{\fM_{\gst}/\fM_{1,1}}$. We will use this to describe $0^!_{\sigma, \loc}[\fC_{\gst}]$ in the next section.

%
%

\section{Calculations}

\def\barW{{\overline W}}
\def\barV{{\overline V}}
\def\barC{\overline C}

\subsection{Proof of the Theorem \ref{main-0}}

Basically our proof follows contents in \cite[Section 4]{LO2}, which proved a similar statement to our Theorem \ref{main-0} in the case of stable map spaces.

Let $M:=\cX\lgst$, and $\pi_{M}:\cC_{M}\to M$ be the universal family. Let $\sL_{M}$ be the universal bundle over $\cC_M$, and
$\sP^i_{M}=\sL_{M}^{\vee\otimes\deg q_i}\otimes \omega_{\cC_{M}/M}$.
By definition the component $\cY\lgst$ is the total space of a vector bundle $\cL$ on $M$, where
$$\cL=\oplus_{i=1}^m\pi_{M*} \sP^i_{M}.
$$
Furthermore, let
$$\gamma: W\defeq \cY\lgst=\Tot(\cL)\lra  M
$$
be the induced (tautological) projection. Here $\Tot(-)$ denote the total space of the bundle.
We denote the bundles
\beq\label{VV}
V'_1=  R^1\pi_{M\ast} \sL_M^{\oplus (n+1)}
,\quad
V'_2=\oplus_{i=1}^mR^1\pi_{M\ast}\sP^i_{M},\and  V' = V'_1\oplus V'_2.
\eeq
By \cite[Proposition 2.8]{L}, we have $H^1(\TT_{M/\fM^{line}_{1,1}})\cong V'_1$.
For any point $x=(C,p_1,\{u_i\})\in M$, we define
\beq
\xi'_1:(V'_1\otimes \cL)|_x\to \CC,\ \ \quad \xi_1(x)(\dot u_i\otimes \chi)=\sum_{i=1}^m\chi_i\sum \frac{\partial q(u)}{\partial u_i}\dot u_i,
\eeq
\begin{eqnarray*}
&&\xi'_2:V'_2|_x\to \CC,\ \ \quad \xi_2(x)(\dot\chi)=\sum_{i=1}^m\dot\chi_i q_i(u), \ \ \quad\chi=(\chi_1,\cdots,\chi_m)\in \Gamma(M,\cL)\\
&& \ \ \quad (\dot u_i)\in V'_1|_x,\ \ \dot\chi=(\dot\chi_1,\cdots,\dot\chi_m)\in V'_2|x.
\end{eqnarray*}

On the other hand, let $\pi_{\cY}:\cC_{\cY}\to \cY$ be the universal family, and $\sL_{\cY}$ be the universal line bundle over $\cC_{\cY}$. Denote $\pi_{W}:\cC_{W}\to W$, $\sP^i_{W}=\sP^i_{\cY}|_{W}$ and $\sL_{W}:=\sL_{\cY}|_{\cC_{W}}$. Recall that the dual perfect obstruction theory of $\cY/\fM^{line}_{1,1} \,$ is $\, \EE\dual_{\cY/\fM^{line}_{1,1}}=R\bul\pi_{\cY\ast}(\sL_\cY^{\oplus (n+1)}\bigoplus \oplus_{i=1}^m\sP^i_\cY)$.
We let
\beq\label{tiV-2}
\widetilde{V}'_1=H^1\left(R\bul\pi_{W\ast}\sL_{W}^{\oplus (n+1)}\right)\cong \gamma\sta V'_1,\ \
\widetilde{V}'_2=H^1(\oplus_iR\bul\pi_{W\ast}\sP^i_{W})\cong \gamma\sta V'_2,
\eeq
and $ \widetilde{V}'=\widetilde{V}'_1\oplus \widetilde{V}'_2$. Both $\widetilde{V}'_1$ and $\widetilde{V}'_2$ are locally free on
$W$.

Denote $\ti\xi'=(\ti\xi'_1,\ti\xi'_2)$, where $\ti\xi'_1:=\gamma^*(\xi'_1)(\cdot\otimes \epsilon)$, $\epsilon\in \Gamma(W,\gamma^{*}\cL)$ is the tautological section and $\ti\xi'_2:=\gamma^*(\xi'_2)$. Then we have
\beq\label{xi-2}
\ti\xi'=\si'|_{\cY_{\text{gst}}}: \widetilde{V}'\lra \sO_{W},
\eeq
where $\si'$ is the cosection defined in \eqref{cos}.
Next we consider the obstruction theory over the Artin stack $\mathfrak{M}_{1,1}$. Moreover we denote
\beq\label{VV}
V_1=  R^1\pi_{E\ast} f_{E}^* T_{\PP^n} \cong H^1(\mathbb{E}\dual_{\cX/\fM_{1,1}} |_{\cX_{\gst}} )
,\quad
V_2=\oplus_{i=1}^mR^1\pi_{M\ast}\sP^i_{M},\and  V=V_1\oplus V_2
\eeq
where $\pi : \cC \to M$ is the universal curve, $\cC_E \subset \cC$ is the universal family of minimal genus 1 subcurves, $\pi_E : \cC_E \to M $ is the projection morphism, and $f_E : \cC_E \to \PP^n$ is the universal morphism. They are vector bundles (locally free sheaves) on $M$ (c.f. Proposition \ref{loc}). Let $\widetilde{V}_1:=\gamma^*V_1$ and $\widetilde{V}_2:=\gamma^*V_2$. Then $H^1(\mathbb{E}_{\cY/\fM_{1,1}} |_{\cY_{\gst}} )=\widetilde{V}:=\widetilde{V}_1\oplus \widetilde{V}_2$. Then the cosection $\ti\xi'=(\ti\xi'_1,\ti\xi'_2)$ induces the cosection $\ti\xi=(\ti\xi_1,\ti\xi_2):\widetilde{V}=\widetilde{V}_1\oplus \widetilde{V}_2\lra \sO_{W}$.

Following \cite[Proposition 3.2]{L}, the non-surjective locus
$D(\ti\xi)$  of $\ti\xi=\si|_{\cY\lgst}$ is
$$D(\si)\times_{\cX}M=Q_{1,1}(X,d)\times_{Q_{1,1}(\PP^n,d)}M,
$$
which is proper.
Let
$$\widetilde{\VV}_1=h^1/h^0(\mathbb{E}\dual_{\cY/\fM_{1,1}} |_{\cX_{\gst}} ),\ \
\widetilde{\VV}_2=h^1/h^0(\oplus_iR\bul\pi_{W\ast}\sP^i_{W}),\ \  \widetilde{\VV}=\widetilde{\VV}_1 \times_W \widetilde{\VV}_2
$$
be the vector bundle stacks.
Then there is a canonical morphism $\rho_j: \widetilde{\VV}_j\to \widetilde{V}_j$ from the bundle stack to its coarse moduli space, for $j=1,2$.
Note that both $\rho_j$ are proper morphisms.

By the base change property of the $h^1/h^0$-construction, and by the definition of $\mathfrak{C}_{\gst}$, we have 
$$[\mathfrak{C}_{\gst}]\in Z\lsta \widetilde{\VV};\quad \widetilde{\VV}=h^1/h^0(\EE\dual_{\cY/\fM_{1,1}}|_{W}).
$$
Let $C\lgst$ be the coarse moduli of $\mathfrak{C}_{\gst}$ relative to $W$,
thus $C\lgst\sub \widetilde{V}$ since $\widetilde{V}$ is the coarse moduli of
$\widetilde{\VV}$.  Further, since the projection $\rho\defeq \rho_1 \times_W \rho_2: \widetilde{\VV}\to \widetilde{V}$
is smooth, we have an identity of cycles
$\rho^\ast[C\lgst]= [\mathfrak{C}_{\gst}]\in Z\lsta \widetilde{\VV}$. Finally, because $[\mathfrak{C}_{\gst}]\in Z\lsta \widetilde{\VV}(\si)$, we have
$$[C\lgst]\in Z\lsta \widetilde{V}(\ti\xi).
$$
Therefore we have the following identity.

\begin{prop}\label{prop-a}\cite[Proposition 6.3]{CL3}
$$0^!_{\si,\loc}[\mathfrak{C}_{\gst}]=0^!_{\ti \xi,\loc}[C\lgst]\in A\lsta D(\ti \xi). 
$$
\end{prop}

Now we calculate the cycle $0^!_{\ti \xi,\loc}[C\lgst]$. We first introduce the following notations
\begin{itemize}
\item[$\diamond$]
$\overline{W} := \PP( \cL \oplus \cO_M)$ be a completion of $W = \cY\lgst$,
\item[$\diamond$]
$\bar{\gamma} : \overline{W} \to \cX\lgst$ be the projection morphisms,
\item[$\diamond$]
Let $D_{\infty} := \PP(\cL \oplus 0) \subset \overline{W}$ be the divisor at infinity,
\item[$\diamond$]
$\overline{V}_1 := \overline{\gamma}^*V_1(-D_{\infty})$, $\overline{V}_2:= \bar{\gamma}^*V_2$, $\overline{V} := \overline{V}_1\oplus \overline{V}_2$,
\item[$\diamond$]
$\bar{\xi}_1 : \overline{V}_1 \to \cO_{\overline{W}}$ and $\bar{\xi}_2 : \overline{V}_2 \to \cO_{\overline{W}}$ are cosections induced from $\ti\xi_1$ and $\ti\xi_2$ respectively,
\item[$\diamond$]
$\bar{\xi} := \bar{\xi}_1 \oplus \bar{\xi}_2$, $\bar{\xi} : \overline{V} \to \cO_{\overline{W}}$.
\end{itemize}


To calculate $\Gysin{\bar{\xi}, \loc}[C\lgst]$, we approximate the cone $C\lgst$ as a subvector bundle of $\overline{V}_1$. To do this, we consider $R := C_{ C_{b,\gst} / C\lgst }$ where $C_{b,\gst} := C\lgst \cap \Tot(0 \oplus \widetilde{V}_2)$. It is a deformation of $C_{\gst}$.

We can easily check that $R$ is embedded in $\Tot(\widetilde{V})$ and $[C\lgst] = [R]$ in $A_*(\widetilde{V}(\ti \xi))$.
Next we investigate the cone $R$ and its completion $\bar{R}$ in $\Tot(\bar{V})$. Similar to \cite[(4.5)]{LO2}, by using a local computation we can check
\[
C_{b,\gst} \subset 0_{\cX\lgst} \cup \gamma^* F
\]
where $0_{\cX\lgst} \subset \cY \lgst = \Tot(\cL)$ is the zero section of the bundle $\cL$, $\Delta_{\cX} := \cX\lgst \cap \cX _{\rede}$ and $F$ is a rank $m$ subbundle of $V_2 |_{\Delta_{\cX}}$ defined in the below.

Recall the quasi-isomorphism
\begin{align*}
\oplus_{i=1}^mR^\bullet\pi_{\cX\ast}\sP^i_{\cX} \stackrel{\mathrm{loc} }{\simeq} \, \bigoplus_{i=1}^m\bigg(\left[ \cO_{\cX} \stackrel{\times t}{\longrightarrow} \cO_{\cX} \right]\oplus \left[ 0 \longrightarrow \cO_{\cX}^{\oplus d \cdot \deg q_i} \right]\bigg),
\end{align*}
we observe that $\left. H^1 \left( \oplus_{i=1}^m R^\bullet\pi_{\cX\ast}\sP^i_{\cX} \right)_{\mathrm{tor} } \right|_{M}$ is a rank $m$ subbundle of $V_2$.
Then we define
$F := \left. H^1 \left( \oplus_{i=1}^mR^\bullet\pi_{\cX\ast}\sP^i_{\cX} \right)_{\mathrm{tor} } \right|_{\Delta_{\cX}} \subset V_2|_{\Delta_{\cX}}.
$

Since $R$ is a cone over $C_{b,\gst} \subset \Tot(\ti V_2)$, we can write
\begin{align*}
[R] = [R_1] + [R_2] \in A_*\left(\Tot(\wtil{V} ) \right)
\end{align*}
where $R_1 := R|_{0_{\cX\lgst}}$ and $[R_2]$ is a cycle supported on $\Tot(\gamma^* F)$.
Parallel to \cite[Lemma 8.1]{CL3} and \cite[p. 24]{LO}, we can check that
\begin{align*}
\Gysin{\bar{\xi},\loc}[R_2] = 0
\end{align*}
since $\dim \Tot(\gamma^* F)$ is smaller than the degree of $[R_2] \in A_*(\Tot(\wtil{V}))$.

Hence we have
\begin{align}\label{eq:locGysin1}
\Gysin{\bar{\xi},\loc}[C\lgst] = \Gysin{\bar{\xi},\loc}[R] = \Gysin{\bar{\xi},\loc}[R_1].
\end{align}
Moreover, by \cite[Proposition 5.3]{LO}, we have
\begin{align}\label{eq:Gysinmapcomp1}
\Gysin{\tilde{\xi}, \loc }[R_1] = \bar{\gamma}_* \Gysin{\bar{\xi}_2,\loc} \cdot \Gysin{\bar{V}_1}[\overline{R}_1]
\end{align}
where $\overline{R}_1 $ is the closure of $R_1$ in $\Tot(\overline{V} )$.

Next we investigate the cone $R_1$. Using a local computation of $R_1$ similar to \cite[(4.8),(4,9)]{LO2}, we conclude that $R_1$ is of the form $R_1 = \gamma^* R_1'$. Here $R_1'$ is given as the closure of the image of the natural composition morphism
\begin{align*}
\varphi : \mathfrak{g}_{\cX_{\gst}}^*N_{\fM_{\gst}/\fM_{1,1}} \cong C_{\cX_{\gst}/\fM_{1,1}} \to C_{\cX / \fM_{1,1}}|_{\cX\lgst} \hookrightarrow V.
\end{align*}
\footnote{Caution : $\varphi$ is similarly defined as $\phi : C_{\cY_{\gst} / \fM_{1,1} }$. But it is slightly different.}


Similar to the local description of $\phi$ in \eqref{localmorph}, we can locally describe $\varphi$ locally as follows:
$$
\varphi|_{\cU'} : 1 \mapsto (w_1,\dots,w_n,0,\dots,0)
$$
over some sufficiently small neighbourhood $\cU' \subset \cX$. From this, we observe the degeneracy locus of the morphism $\varphi$ is $\Delta_{\cX}$. To resolve this, we consider the blow-up
\[
\hat{M} : = \mathrm{Bl}_{\Delta_{\cX}} M, \ p : \hat{M} \to M.
\]
Let $E$ be the exceptional divisor. Then there is an induced morphism
\[
\hat{\varphi} : \left( p^* \mathfrak{g}_{\cX_{\gst}}^*N_{\fM_{\gst}/\fM_{1,1}} \right)(E) \to p^* V
\]
which is an injective morphism of vector bundles. Thus its image $\image(\hat{\varphi})$ is a line subbundle of $p^* \bar{V}$. Then we have
\[
p( \Tot( \image(\hat{\varphi}) ) ) = R_1'.
\]
There is the following induced morphism
\begin{align*}
\bar{\varphi} : \left( \hat{\gamma}^* \left( p^* \mathfrak{g}_{\cX_{\gst}}^*N_{\fM_{\gst}/\fM_{1,1}} \right)(E) \right)(-q^*D_\infty) \to q^* \bar{V}
\end{align*}
where $q : \hat{ W } := \overline{W} \times_{M} \hat{M} \to \overline{W} $ is the projection, $\hat{\gamma} : \hat{W} \to \hat{M}$ is the projection. Note that $\bar{\varphi}$ is an injective morphism of vector bundles.
We have
\[
q(\Tot(\image(\bar{\varphi}))) = \overline{R}_1.
\]
Then we obtain
\begin{align*}
\Gysin{\bar{V}_1}[\overline{R}_1] =
q_* \Gysin{q^*\bar{V}}[\Tot(\image(\bar{\varphi}))] = q_*\left(
 c_{\mathrm{top}}(q^*\bar{V} / \image(\bar{\varphi})) \cap [\hat{M}] \right).
\end{align*}
Hence, by combining the above computation with \eqref{eq:locGysin1} and \eqref{eq:Gysinmapcomp1}, we have
\begin{align*}
\Gysin{\bar{\xi},\loc}[C\lgst] = \bar{\gamma}_* \Gysin{\bar{\xi}_2,\loc} \cdot \Gysin{\bar{V}_1}[\overline{R}_1] =
\Gysin{\xi_2,\loc} \left( \bar{\gamma}_*\Gysin{\bar{V}_1}[\overline{R}_1] \right) = \Gysin{\xi_2,\loc} \left( \bar{\gamma}_* q_*\left(
 c_{\mathrm{top}}(q^*\bar{V} / \image(\bar{\varphi})) \cap [\hat{M}] \right) \right)
\end{align*}
where the second equality comes form the functorial property of localized Gysin homomorphisms \cite{KL}.

By using \cite[Lemma 4.2]{LO2} and \cite[(4.13)]{LO2}, we have
\begin{align*}
\bar{\gamma}_* q_*\left(
 c_{\mathrm{top}}(q^*\bar{V} / \image(\bar{\varphi})) \cap [\hat{M}] \right) = \left( \frac{ c(V_1)s(\cL\dual) }{c( \mathfrak{g}_{\cX_{\gst}}^*N_{\fM_{\gst}/\fM_{1,1}}  )} \right)_{\rank V_1 - m-1} \cap [M].
\end{align*}
Therefore we have
\begin{align*}
\Gysin{\bar{\xi},\loc}[C\lgst] = \left( \frac{ c(V_1)s(\cL\dual) }{c( \mathfrak{g}_{\cX_{\gst}}^*N_{\fM_{\gst}/\fM_{1,1}}  )} \right)_{\rank V_1 - m-1} \cap \, \Gysin{\xi_2, \loc}[M].
\end{align*}
Note that $M$ is considered as a substack of $\Tot(V_2)$ embedded by the zero section.

\bigskip

Next, consider the node-identifying morphism
\begin{align*}
\iota : \overline{M}_{1,1} \times Q_{0,2}(\PP^n,d) & \to Q_{1,1}(\PP^n,d) = \cX
\end{align*}

Let $\cH$ be the Hodge bundle over $\overline{M}_{1,1}$, $L_1$ be the universal tangent bundle over $\overline{M}_{1,1}$ at the marked point, $L_2$ be the universal tangent bundle over $Q_{0,2}(\PP^n,d)^p$ at the second marked point, which comes from splitting of the node. We have $\cH \dual \cong L_1$.
Moreover we have
\begin{itemize}
\item[$\diamond$]
$\iota^*V_1 \cong \cH \boxtimes ev_2^* T_{\PP^n}$,
\item[$\diamond$]
$\iota^* \cL\dual \cong \cH \boxtimes (\oplus_i ev_2^* \cO_{\PP^n}(\deg q_i) )$,
\item[$\diamond$]
$\iota^*\mathfrak{g}_{\cX_{\gst}}^*N_{\fM_{\gst}/\fM_{1,1}} \cong \cH\dual \boxtimes L_2 $,
\item[$\diamond$]
$\iota^{-1}[M] = \overline{M}_{1,1} \times Q_{0,2}(\PP^n,d)$,
\item[$\diamond$]
$\Gysin{\iota^*\xi_2,\loc}([\overline{M}_{1,1}]\times [Q_{0,2}(\PP^n,d)]) = [\overline{M}_{1,1}] \times [Q_{0,2}(X,d)]\virt$.
\end{itemize}

Thus we have
\begin{align*}
\Gysin{\bar{\xi},\loc}[C\lgst] & = \left( \frac{ c(V_1)s(\cL\dual) }{c( \mathfrak{g}_{\cX_{\gst}}^*N_{\fM_{\gst}/\fM_{1,1}}  )} \right)_{\rank V_1 - m-1} \cap \Gysin{\xi_2, \loc}[M] \\
& =  (-1)^{(\sum_i\deg q_i)d}\iota_* \left( \frac{c(\cH\dual \boxtimes ev_2^* T_{\PP^n})s(\cH\dual \boxtimes ev_2^*(  \oplus_i\cO_{\PP^n}(\deg q_i) ) ) }{c(\cH\dual \boxtimes L_2) } \right)_{\rank V_1 - m-1}\\
 &\cap \left( [\overline{M}_{1,1}] \times [Q_{0,2}(X,d)]\virt \right) \\
& = (-1)^{(\sum_i\deg q_i)d}\iota_* \left( \frac{c(\cH\dual \boxtimes ev_2^* T_X)  }{c(\cH\dual \boxtimes L_2) } \right)_{\rank V_1 - m-1} \cap \left( [\overline{M}_{1,1}] \times [Q_{0,2}(X,d)]\virt \right).
\end{align*}
where the last identity comes from the short exact sequence $0 \to T_X \to T_{\PP^n}|_X \to  \oplus_i\cO_{\PP^n}(\deg q_i)|_X \to 0$.
Let us define
\[
A^{\rede }_{1,d} := (-1)^{(\sum_i\deg q_i)d} \, 0_{\si,\loc}^![\mathfrak{C}_{\rede}].
\]
We will call it the virtual cycle for reduced quasi-map invariants. We set
\[
N_{\rede} := \pi_* (\oplus_{i=1}^m \sL_{\cY}^{\deg q_i} |_{\cY_{\rede}} ) = \pi_* (\oplus_{i=1}^m \sL_{\cY}^{\deg q_i} |_{\cX_{\rede}} )
\]
for the universal curve $\pi:=\pi_{\cY}|_{\cY_{\rede}} : \cC_{\cY}|_{ \cY_{\rede}} \to \cY_{\rede}$. Then by Theorem \ref{main}, $N_{\rede}$ is a vector bundle.

In the same manner as in \cite[Section 4.3]{LO2} we can show that
\begin{eqnarray}\label{red-def}
A^{\rede }_{1,d}&=& (-1)^{(\sum_i\deg q_i)d} \,0_{N_{\rede}\dual,s\dual}^![\cY_{\rede}]\in A_{*}(Q_{1,1}(X,d)).
\end{eqnarray}
where $s$ is the natural section $s : \cO_{\cY_{\rede}} \to N_{\rede}$ which is induced from the defining equations $q_1,\dots q_m$ of $X \subset \PP^n$. Let $e^{\mathrm{ref}}(N_{\rede})$ be the refined euler class localized by the section $s$. Note that we have
\[
0^!_{N_{\rede}\dual,s\dual}[\cY_{\rede}] = (-1)^{\rank(N_{\rede})} e^{\mathrm{ref}}(N_{\rede})[\cY_{\rede}] = (-1)^{(-1)^{(\sum_i\deg q_i)d}} e^{\mathrm{ref}}(N_{\rede})[\cY_{\rede}]
\]

By the proof in \cite[Section 5]{CL3}, we have
\begin{eqnarray}\label{form-2}
A^{\rede }_{1,d}&=& (-1)^{(\sum_i\deg q_i)d} \iota_{1*}0_{N_{\rede},s\dual}^![\cY_{\rede}]\\
&=& e^{\mathrm{ref}}\left( \oplus_{i=1}^m\pi_{\cX*}\sL^{\otimes \deg q_i}_{\cX}|_{\cX_{\rede}} \right)\cap \cX_{\rede}.\nonumber
\end{eqnarray}
In summary, we obtain the following
\begin{eqnarray}\label{form-1}
&&[Q_{1,1}(X,d))]^{\virt}\\
&=&(-1)^{(\sum_i\deg q_i)d}[\cY]_{\loc}\virt\nonumber\\
&=&(-1)^{(\sum_i\deg q_i)d}\bigg(0_{\si,\loc}^![\mathfrak{C}\lpri]+\Gysin{\bar{\xi},\loc}[C\lgst]\bigg)\nonumber\\
&=& A^{\rede }_{1,d}+ \nonumber\\
&&\iota_* \left( \frac{c(\cH\dual \boxtimes ev_2^* T_X)  }{c(\cH\dual \boxtimes L_2) } \right)_{n - m-1} \cap \left( \, [\overline{M}_{1,1}] \times [Q_{0,2}(X,d)]\virt \, \right.\nonumber
\end{eqnarray}
where $\iota : \bar{M}_{1,1} \times Q_{0,2}(X,d) \to Q_{1,1}(X,d)$ is the node-identifying morphism. It proves the main Theorem \ref{main-0}.

\subsection{Proof of the Corollary \ref{coro-1}}

Let $X\subset \PP^n$ be a complete intersection with dimension $n-m$, then the virtual dimension
\begin{eqnarray*}\text{vdim}\, Q_{g,k}(X,d)&=&\int_{d[\PP^1]}c_1(T_X)+(1-g)(n-m-3)+k.
\end{eqnarray*}
Let $p_1:\overline{M}_{1,1} \times Q_{0,2}(X,d)\to \overline{M}_{1,1}$ and $p_2:\overline{M}_{1,1} \times Q_{0,2}(X,d)\to Q_{0,2}(X,d)$ be the two projections.
\begin{eqnarray*}
c(\cH\dual \boxtimes ev_2^* T_X)&=&\sum_{r=0}^{n-m}\sum_{i=0}^{r}p_1^*c_1(\cH^{\vee})^{r-i}p_2^*c_i(ev_2^* T_X)\\
&=&1+p_1^*c_1(\cH^{\vee})\bigg(\sum_{i=0}^{n-m-1}p_2^*c_i(ev_2^* T_X)\bigg)+\cdots,
\end{eqnarray*}
where $\cdots$ are the terms such that they contain factor of $c^i_1(\cH^{\vee})$ with $i>1$.
and
\begin{eqnarray*}
\frac{1}{c(\cH\dual \boxtimes L_2)}&=&1+\sum_{i\ge1}(-1)^i(p_1^*c_1(\cH^{\vee})+p_2^*c_1(L_2))^i\\
&=&1+\sum_{i\ge1}(-1)^i\binom{i}{1}p_1^*c_1(\cH^{\vee})p_2^*c_1(L_2)^{i-1}+\cdots.
\end{eqnarray*}

\begin{eqnarray*}
\frac{c(\cH\dual \boxtimes ev_2^* T_X)}{c(\cH\dual \boxtimes L_2)}&=&\bigg(1+p_1^*c_1(\cH^{\vee})\bigg(\sum_{i=0}^{n-m-1}p_2^*c_i(ev_2^* T_X)\bigg)+\cdots\bigg)\\
&&\bigg(1+\sum_{i\ge1}(-1)^i\binom{i}{1}p_1^*c_1(\cH^{\vee})p_2^*c_1(L_2)^{i-1}+\cdots\bigg)\\
&=&1+p_1^*c_1(\cH^{\vee})\bigg(\sum_{i=0}^{n-m-1}p_2^*c_i(ev_2^* T_X)+\sum_{i\ge1}(-1)^i\binom{i}{1}p_2^*c_1(L_2)^{i-1}\bigg)+\cdots.
\end{eqnarray*}

Let $\psi_i$ be the psi class, which is the first Chern class of the universal cotangent line bundle for the $i$-th marking. Let $\gamma \in H^{2k}(X,\QQ)
$ be a cohomology class such that $k \le 1$, and let $a$ be an integer satisfies $a+k=\text{vdim}\, Q_{1,1}(X,d)$. By formula (\ref{form-1}) and (\ref{form-2}), we have the following formula for stable quasimap invariants

\begin{eqnarray*}
\langle\psi^{a}ev^*\gamma \rangle_{1,1,d}&=&\int_{A^{\rede }_{1,d}}\psi^{a}ev^*\gamma+\int_{\Gysin{\bar{\xi},\loc}[C\lgst]}\psi^{a}ev^*\gamma \\
&=&\langle\psi^{a}ev^*\gamma\rangle^{\rede}_{1,1,d}+\int_{\overline{M}_{1,1}}c_1(\cH^{\vee})\bigg(\int_{Q_{0,2}(X,d)}\psi^{a}ev^*\gamma \, c_{n-m-2}(ev_2^* T_X)\nonumber\\
&&+(-1)^{n-m-1}\binom{n-m-1}{1}\int_{Q_{0,2}(X,d)}\psi^{a}ev^*\gamma \, c_1(L_2)^{n-m-2}\bigg)\nonumber\\
&=&\langle\psi^{a}ev^*\gamma\rangle^{\rede}_{1,1,d}-\frac{1}{24}\bigg(\int_{Q_{0,2}(X,d)}\psi^{a}ev^*\gamma \, c_{n-m-2}(ev_2^* T_X)\nonumber\\
&&-(n-m-1)\int_{[{Q_{0,2}(X,d)}]\virt} \psi^{a}ev^*\gamma \, \psi_2^{n-m-2}\bigg).\nonumber
\end{eqnarray*}
Here we denoted $c_1(\cH\dual) = \psi$.
If $X$ is a Calabi-Yau threefold, then $c_1(T_X)=0$, and $n-m=3$. So we obtain
\begin{eqnarray}
\langle\psi^{a}ev^*\gamma\rangle_{1,1,d}&=&\langle\psi^{a}ev^*\gamma\rangle^{\rede}_{1,1,d}+\frac{1}{12}\int_{[{Q_{0,2}(X,d)}]\virt}\psi^{a}ev^*\gamma \, \psi_2. \nonumber
\end{eqnarray}


\bibliographystyle{amsplain}

\end{document}